\def\version{30/05/2011 version 6: to appear in Central
European Journal of Mathematics
\hfill
\href{http://arxiv.org/abs/1005.1662}{arxiv:1005.1662}
}

\documentclass[a4paper,11pt]{amsart}
\usepackage{amssymb,amscd,amsxtra}
\usepackage{fullpage}
\usepackage[all]{xy}
\usepackage{mathrsfs}
\usepackage{hyperref}

\newdir{ >}{{}*!/-8pt/@{>}}

\theoremstyle{plain}
\newtheorem{thm}{Theorem}[section]

\newtheorem{lem}[thm]{Lemma}
\newtheorem{prop}[thm]{Proposition}
\newtheorem{cor}[thm]{Corollary}
\theoremstyle{definition}
\newtheorem{rem}[thm]{Remark}

\numberwithin{equation}{section}

\def\ie{\emph{i.e.}}

\def\:{\colon}
\def\.{\cdot}
\def\<{\langle}
\def\>{\rangle}
\def\({\left(}
\def\){\right)}
\def\ph#1{\phantom{#1}}
\def\epsilon{\varepsilon}
\def\phi{\varphi}
\def\leq{\leqslant}
\def\geq{\geqslant}

\def\la{\leftarrow}
\def\lra{\longrightarrow}
\def\Lra{\Longrightarrow}
\def\ra{\rightarrow}

\def\iso{\cong}

\def\F{\mathbb{F}}

\def\Q{\mathbb{Q}}

\def\Z{\mathbb{Z}}

\def\ideal{\triangleleft}

\DeclareMathOperator{\rad}{rad}

\DeclareMathOperator{\Tor}{Tor}
\def\id{\mathrm{id}}

\def\nr{\mathrm{nr}}

\def\Enr{E^\nr}

\DeclareMathOperator{\soc}{soc}
\DeclareMathOperator{\gldim}{gl\ph{-}dim}
\DeclareMathOperator{\projdim}{proj\ph{-}dim}
\DeclareMathOperator{\res}{res}
\DeclareMathOperator{\Cotor}{Cotor}

\begin{document}
\title
[Lubin-Tate cohomology of classifying spaces]
{Some properties of Lubin-Tate cohomology for classifying
spaces of finite groups}
\author{Andrew Baker \and Birgit Richter}
\address{School of Mathematics \& Statistics, University
of Glasgow, Glasgow G12 8QW, Scotland.}
\email{a.baker@maths.gla.ac.uk}
\urladdr{\href{http://www.maths.gla.ac.uk/~ajb}
                       {http://www.maths.gla.ac.uk/$\sim$ajb}}
\address{Fachbereich Mathematik der Universit\"at Hamburg,
Bundesstrasse 55, 20146 Hamburg, Germany.}
\email{birgit.richter@uni-hamburg.de}
\urladdr{\href{http://www.math.uni-hamburg.de/home/richter/}
              {http://www.math.uni-hamburg.de/home/richter/}}
\date{\version}
\thanks{The second author was supported by the Glasgow Mathematical
Journal Learning and Research Support Fund, and she thanks the
Mathematics Department of the University of Glasgow for its
hospitality. We would like to thank Tilman Bauer, Ken Brown, Jim
Davis, Uli Kr\"ahmer and Jesper M\o ller for helpful comments.}
\begin{abstract}
We consider brave new cochain extensions $F(BG_+,R)\lra F(EG_+,R)$,
where $R$ is either a Lubin-Tate spectrum $E_n$ or the related
$2$-periodic Morava K-theory $K_n$, and $G$ is a finite group.
When $R$ is an Eilenberg-Mac~Lane spectrum, in some good cases
such an extension is a $G$-Galois extension in the sense of
John Rognes, but not always faithful. We prove that for $E_n$
and $K_n$ these extensions are always faithful in the $K_n$
local category. However, for a cyclic $p$-group $C_{p^r}$, the
cochain extension $F({BC_{p^r}}_+,E_n) \lra F({EC_{p^r}}_+,E_n)$
is not a Galois extension because it ramifies. As a consequence,
it follows that the $E_n$-theory Eilenberg-Moore spectral sequence
for $G$ and $BG$ does not always converge to its expected target.
\end{abstract}

\maketitle

\section{Introduction}

In the algebraic Galois theory of commutative rings~\cite{CHR},
faithful flatness is a property implied by separability. However,
in the topological analogue, the brave new Galois theory of
Rognes~\cite{JR:Opusmagnus}, this is not true. The simplest
counterexample, due to Ben Wieland~\cite{JR:unfaithful}, is
provided by the $C_2$-Galois extension
\begin{equation} \label{eq:wieland}
F(B{C_2}_+,H\F_2) \lra F(E{C_2}_+,H\F_2) \sim H\F_2
\end{equation}
which is not faithful. This example relies on the algebraic
fact that
\[
\pi_*(F(B{C_2}_+,H\F_2))=H^{-*}(BC_2;\F_2)
\]
is a polynomial algebra and so has finite global dimension.

In this note we consider this question for a Lubin-Tate
spectrum $E_n$ and the related Morava $K$-theory $K_n$,
and show that for any finite group~$G$, the extension
\begin{equation}\label{eq:En-cochains}
E_n^{BG}=F(BG_+,E_n) \lra F(EG_+,E_n) \sim E_n
\end{equation}
is faithful as an $E_n$-module. We also show that the
non-commutative extension
\begin{equation}\label{eq:Kn-cochains}
F(BG_+,K_n) \lra F(EG_+,K_n) \sim K_n
\end{equation}
is faithful and $F(BG_+,K_n)$ is a faithful $E_n$-module.
A crucial difference from $F(BG_+,H\F_p)$ is that $K_n^*(BG_+)$
is always an Artinian algebra over $(K_n)_*$, and so if
$K_n^*(BG_+)\neq K_n^*$ then it has infinite global dimension by
Proposition~\ref{prop:Artin-dim}.

Our approach to this involves introducing an analogue of the
algebraic socle series for a module over an Artinian ring, and
we show that this behaves well enough to prove our result.

We show in Section~\ref{sec:E^BG&GaloisThy} that for a cyclic
$p$-group $C_{p^r}$, the cochain extension
$F({BC_{p^r}}_+,E_n) \lra F({EC_{p^r}}_+,E_n)$ is ramified and
hence it is not a Galois extension. As a consequence it follows
that the $E_n$-theory Eilenberg-Moore spectral sequence for such
groups does not converge to its expected target, whereas work
of Tilman Bauer indicates that this is not the case for Morava
$K$-theory.

\subsection*{Notation, etc}
In discussing purely algebraic notions we will often use boldface
symbols $\boldsymbol{A},\boldsymbol{M},\ldots$ to denote rings,
modules, etc, while for topological objects such as $S$-algebras
and their modules we will use italic symbols $A,M,\ldots$, thereby
hopefully reducing the possibility of confusion between the two
settings. For an associative $S$-algebra $A$, we denote by
$\mathscr{D}_{A}$ the derived category of $A$-module spectra defined
in~\cite[chapter~III, construction~ 2.11]{EKMM}.

We follow Lam~\cite[theorem 19.1]{Lam:Noncomm} in using the
phrase \emph{local ring} to indicate a ring with a unique
maximal left ideal (necessarily $2$-sided and equal to its
Jacobson radical); the quotient of such a ring by its Jacobson
radical is a division ring. For non-commutative rings other
terminology is often encountered such as \emph{scalar local
ring}.

\subsection*{Brave new Galois extensions}

The following definition of a Galois extension is due to John
Rognes~\cite{JR:Opusmagnus}. Let $A$ be a commutative $S$-algebra
and let $B$ be a commutative cofibrant $A$-algebra. Let $G$ be
a finite (discrete) group and suppose that there is an action
of $G$ on $B$ by commutative $A$-algebra morphisms. Then $B/A$
is a \emph{$G$-Galois extension} if it satisfies the following
two conditions:
\begin{itemize}
\item
The natural map
\[
A\lra B^{hG}=F(EG_+,B)^G
\]
is a weak equivalence of $A$-algebras.
\item
There is a natural equivalence of $B$-algebras
\[
\Theta\:B\wedge_A B\xrightarrow{\sim}F(G_+,B)
\]
induced from the action of $G$ on the right hand
factor of~$B$.
\end{itemize}
Furthermore, $B/A$ is a \emph{faithful $G$-Galois extension}
if it also satisfies
\begin{itemize}
\item
$B$ is faithful as an $A$-module, \ie, for any $A$-module $M$,
$B\wedge_A M \sim *$ implies that $M \sim *$.
\end{itemize}

Examples like \eqref{eq:wieland}  show that not every Galois
extension is faithful.

\section{Recollections on modules over Artinian algebras}
\label{sec:ArtinAlg&mod}

In this section we review some standard algebraic background
material; good sources for this are~\cite{Alperin,Lam:Noncomm}.

Let $\boldsymbol{D}$ be a division ring. A ring $\boldsymbol{A}$
equipped with homomorphisms of rings
$\eta\:\boldsymbol{D}\lra\boldsymbol{A}$ and
$\epsilon\:\boldsymbol{A}\lra\boldsymbol{D}$ is an \emph{augmented
$\boldsymbol{D}$-algebra} if the following diagram commutes.
\[
\xymatrix{
\boldsymbol{D} \ar[rr]^=\ar[dr]_\eta&& \boldsymbol{D} \\
&\boldsymbol{A}\ar[ur]_\epsilon&
}
\]
The augmentation $\epsilon$ splits the unit $\eta$. We will
also say that $\boldsymbol{A}$ is an \emph{Artinian local}
$\boldsymbol{D}$-algebra if it is Artinian and local.

If $\boldsymbol{A}$ is an Artinian local augmented
$\boldsymbol{D}$-algebra, then the  Jacobson radical of
$\boldsymbol{A}$ is
\[
\boldsymbol{J} = \rad(\boldsymbol{A}) = \ker\epsilon.
\]
By~\cite[theorem~4.12]{Lam:Noncomm}, $\boldsymbol{J}$ is
nilpotent, say $\boldsymbol{J}^{e}=0$ and $\boldsymbol{J}^{e-1}\neq0$.
\begin{lem}\label{lem:ArtinAlg&mod}
Let $\boldsymbol{A}$ be as above and let $\boldsymbol{M}$
be a left $\boldsymbol{A}$-module.  If\/
$\boldsymbol{D}\otimes_{\boldsymbol{A}}\boldsymbol{M}=0$,
then $\boldsymbol{M}=0$.
\end{lem}
\begin{proof}
Comparing the two horizontal exact sequences
\[
\xymatrix{
&\boldsymbol{J}\otimes_{\boldsymbol{A}} \boldsymbol{M} \ar[r]\ar[d]
&\boldsymbol{A}\otimes_{\boldsymbol{A}} \boldsymbol{M} \ar[r]\ar[d]^{\iso}
& \boldsymbol{D}\otimes_{\boldsymbol{A}} \boldsymbol{M} \ar[r]\ar[d]^{\iso} & 0 \\
0\ar[r] &\boldsymbol{J}\boldsymbol{M} \ar[r] & \boldsymbol{M} \ar[r]
& \boldsymbol{M}/\boldsymbol{J}\boldsymbol{M} \ar[r] & 0
}
\]
we see that if $\boldsymbol{D}\otimes_{\boldsymbol{A}}\boldsymbol{M}=0$
then
\[
\boldsymbol{M}=\boldsymbol{J}\boldsymbol{M}
              =\ldots=\boldsymbol{J}^{e}\boldsymbol{M}=0.
\qedhere
\]
\end{proof}

Let $\boldsymbol{M}$ be a left $\boldsymbol{A}$-module.
The \emph{socle} of $\boldsymbol{M}$ is the submodule
\[
\soc^1\boldsymbol{M} = \soc\boldsymbol{M}
= \{x\in\boldsymbol{M}:\boldsymbol{J}x=0\},
\]
which can also be characterized as the sum of all the
simple $\boldsymbol{A}$-submodules of $\boldsymbol{M}$.
The \emph{socle series} of $\boldsymbol{M}$ is the
increasing sequence of submodules
\begin{equation*}
0=\soc^0\boldsymbol{M}\subseteq\soc^1\boldsymbol{M}
  \subseteq\ldots\subseteq\soc^k\boldsymbol{M}\subseteq
  \soc^{k+1}\boldsymbol{M}\subseteq\ldots\subseteq\boldsymbol{M},
\end{equation*}
where for each $k$ the following is a pullback square
\begin{equation*}
\xymatrix{
\soc^{k+1}\boldsymbol{M} \ar[r]\ar[d]
           & \soc(\boldsymbol{M}/\soc^k\boldsymbol{M})\ar[d] \\
\boldsymbol{M} \ar[r] & \boldsymbol{M}/\soc^k\boldsymbol{M}
}
\end{equation*}
so we have
\begin{equation*}
\soc^k\boldsymbol{M} = \{x\in\boldsymbol{M}:\boldsymbol{J}^kx=0\},
\end{equation*}
and
\begin{equation*}
\soc^e\boldsymbol{M} = \boldsymbol{M}.
\end{equation*}
In fact, for small $k$
\begin{equation*}
\soc^k\boldsymbol{M}\subset\soc^{k+1}\boldsymbol{M},
\end{equation*}
until we reach a value $k=k_0\leq e$ for which
$\soc^{k_0}\boldsymbol{M}=\boldsymbol{M}$.

It is also clear that given a homomorphism
$\phi\:\boldsymbol{M}\lra\boldsymbol{N}$ of $\boldsymbol{A}$-modules
there are compatible homomorphisms
\[
\soc^k\boldsymbol{M}\lra\soc^k\boldsymbol{N}.
\]
For details on the socle series see~\cite{Lam:Noncomm},
especially Ex.~4.18, and~\cite[chapter~I, section~1]{Alperin}.

We end this section with a result that supplies an algebraic
backdrop for some of our later work. We give a proof suggested
by K.~Brown.
\begin{prop}\label{prop:Artin-dim}
Let $\boldsymbol{A}$ be a local left-Artinian ring which is
not a division ring. Then
\[
\projdim (\boldsymbol{A}/\rad(\boldsymbol{A}))
                          = \gldim\boldsymbol{A} = \infty,
\]
where $\boldsymbol{A}/\rad(\boldsymbol{A})$ is the unique
simple left $\boldsymbol{A}$-module.
\end{prop}
\begin{proof}
Since $\boldsymbol{A}$ is local, it has only one simple
module and therefore
\[
\projdim (\boldsymbol{A}/\rad(\boldsymbol{A}))
                                = \gldim \boldsymbol{A}.
\]
Also, since $\boldsymbol{A}$ is Artinian it has a left
ideal~$\boldsymbol{I}$ isomorphic to
$\boldsymbol{A}/\rad(\boldsymbol{A})$. The corresponding
exact sequence
\begin{equation}\label{eq:Artin-dim}
0\ra \boldsymbol{I}\lra \boldsymbol{A}
                   \lra \boldsymbol{A}/\boldsymbol{I} \ra 0
\end{equation}
cannot split since $\boldsymbol{A}$ is local and therefore
it has no non-trivial idempotents.

If
\begin{equation*}
\projdim (\boldsymbol{A}/\rad(\boldsymbol{A}))
                    = \gldim\boldsymbol{A} < \infty,
\end{equation*}
then~\eqref{eq:Artin-dim} would give
\begin{equation*}
\projdim (\boldsymbol{A}/\rad(\boldsymbol{A})) + 1
   = \projdim (\boldsymbol{A}/\boldsymbol{I})
   \leq \gldim \boldsymbol{A}
   = \projdim (\boldsymbol{A}/\rad(\boldsymbol{A})),
\end{equation*}
which is impossible.
\end{proof}

\begin{rem}\label{rem:gradings}
We end this section by noting that the above discussion works
as well if we assume that $\boldsymbol{A}$ is graded, provided
this is suitably interpreted. In our work below we are interested
in $\Z$-gradings which are also $2$-periodic, \ie, for all $n\in\Z$,
$(-)_{n+2}=(-)_n$. This can be interpreted as a $\Z/2$-grading.
\end{rem}

\section{Socle series in topology}\label{sec:Top-Socle}

Let $D$ be an $S$-algebra for which
$\pi_0D$ is a non-trivial
division ring, $\pi_1D=0$, and the graded ring $\pi_*D=\boldsymbol{D}$
has period two.
Suppose that $A$ is an $S$-algebra
both under and over $D$, giving the following diagram of morphisms
of $S$-algebras.
\begin{equation}\label{eq:D-A-D}
\xymatrix{
D \ar[rr]^{=}\ar[rd]_\eta && D \\
& A\ar[ru]_\epsilon &
}
\end{equation}

We assume that $\boldsymbol{A}=\pi_*A$ is an Artinian local augmented
$\boldsymbol{D}$-algebra, so that the augmentation ideal $\ker\epsilon$
is the Jacobson radical of $\boldsymbol{A}$, $\rad(\boldsymbol{A})$,
and also $\rad(\boldsymbol{A})^e=0$ and $\rad(\boldsymbol{A})^{e-1}\neq0$.

\begin{rem}\label{rem:soc-Dmod}
Let $M$ be a left $A$-module. Then $\boldsymbol{M}=\pi_*M$ is a
left $\boldsymbol{A}$-module and its socle $\soc\boldsymbol{M}$
is a $\boldsymbol{D}$-module through both the unit $\eta$ and
the augmentation $\epsilon$, and these module structures agree
since $\rad(\boldsymbol{A})=\ker\epsilon$.
\end{rem}

\begin{thm}\label{thm:TopSocle}
There are functors $\soc^k\:\mathscr{D}_{A}\lra\mathscr{D}_{A}$
for $0\leq k\leq e$ such that \\
\emph{(a)} for each $k$, $\pi_*(\soc^k M)=\soc^k\boldsymbol{M}$; \\
\emph{(b)} there are natural transformations $\soc^k M\lra\soc^{k+1}M$
giving a commutative diagram
\[
\xymatrix{
0\ar[r] & \pi_*\soc^1 M\ar[r]\ar[d]^\iso & \pi_*\soc^2 M\ar[r]\ar[d]^\iso
 & {\ldots} \ar[r]&\pi_*\soc^e M\ar[r]\ar[d]^\iso &0 \\
0\ar[r] & \soc^1 \boldsymbol{M}\ar[r] & \soc^2\boldsymbol{M}\ar[r]
 & {\ldots} \ar[r] &\soc^e \boldsymbol{M}\ar[r] &0
}
\]
which is natural with respect to morphisms of $A$-modules.
\end{thm}
\begin{proof}
As $\boldsymbol{D}$ is a graded division ring, $\soc\boldsymbol{M}$
is a $\boldsymbol{D}$-vector space. Since $M$ is a $D$-module
via the unit we can find a morphism of $D$-modules
\begin{equation}\label{eq:socrealisation}
\bigvee_j\Sigma^{s(j)}D \lra M
\end{equation}
to realize an algebraic isomorphism
\begin{equation*}
\bigoplus_jD_{*-s(j)} \xrightarrow{\;\iso\;} \soc\boldsymbol{M}
                            \subseteq \boldsymbol{M}.
\end{equation*}
Now Remark~\ref{rem:soc-Dmod} implies that the morphism
of~\eqref{eq:socrealisation} is actually one of $A$-modules.
We set $\soc M=\bigvee_j\Sigma^{s(j)}D$.

Now we can repeat this on the cofibre $M/\soc M$ of the map
$\soc M\lra M$, obtaining $\soc(M/\soc M)\lra M/\soc M$. We
then define $\soc^2M$ using the right hand pullback square
in the diagram
\[
\xymatrix{
\soc M\ar[r]\ar[d]_= & \soc^2M\ar[d]\ar[r] & \soc(M/\soc M)\ar[d] \\
\soc M\ar[r] & M\ar[r] & M/\soc M
}
\]
from which we see by a standard diagram chase that
$\pi_*(\soc^2 M)\iso\soc^2\boldsymbol{M}$.
Continuing in this way we inductively build the socle tower
\[
*\ra \soc^1 M \lra \soc^2M\lra
                   \ldots \lra \soc^{e-1}M\lra \soc^eM = M,
\]
using pullback squares
\[
\xymatrix{
\soc^{k+1}M\ar[d]\ar[r] & \soc(M/\soc^k M)\ar[d] \\
M\ar[r] & M/\soc^k M
}
\]
for each $k$. These satisfy
\[
\pi_*(\soc^k M) = \soc^k\boldsymbol{M}.
\qedhere
\]
\end{proof}

An important consequence of this construction is that
there is a minimal $k_0$ for which $\soc^{k_0}M = M$,
so since $\soc^{k_0-1}\boldsymbol{M}\neq\boldsymbol{M}$,
using the fibre sequence
\begin{equation}\label{eq:M->M''}
\soc^{k_0-1}M \lra M \lra M/\soc^{k_0-1}M,
\end{equation}
we obtain $\pi_*(M/\soc^{k_0-1}M)\neq0$.

\begin{lem}\label{lem:Non-triv-D}
The $A$-module $D$ satisfies $\pi_*(D\wedge_AD)\neq0$.
\end{lem}
\begin{proof}
There is a diagram of left $D$-modules induced from~\eqref{eq:D-A-D}
\[
\xymatrix{
D\wedge_DD \ar[rr]^{=}\ar[rd] && D\wedge_DD \\
& D\wedge_AD\ar[ru] &
}
\]
in which $D\wedge_DD\iso D$. On applying $\pi_*(-)$
we see that $\pi_*(D\wedge_AD)\neq0$.
\end{proof}

\begin{thm}\label{thm:Non-triv}
Let $M$ be an $A$-module for which $\pi_*M\neq0$. Then
$\pi_*(D\wedge_AM)\neq0$, \ie, $D$ is a faithful $A$-module.
\end{thm}
\begin{proof}
Using the socle series we can find a fibration sequence
as in~\eqref{eq:M->M''},
\begin{equation}\label{eq:M-fibseq}
M'\lra M \lra M'',
\end{equation}
where $\boldsymbol{M}''=\pi_*M''\neq0$, $J\boldsymbol{M}''=0$
and there is a short exact sequence
\begin{equation}\label{eq:M-fibseq-htpy}
0\ra\pi_*(M')\lra \pi_*(M) \lra \pi_*(M'')\ra0.
\end{equation}
As remarked in the proof of Theorem~\ref{thm:TopSocle}, $M''$
is weakly equivalent to a wedge of copies of suspensions of
the $A$-module $D$. So $\pi_*(M'')$ is a direct sum of copies
of suspensions of $\pi_*(D)$, hence by Lemma~\ref{lem:Non-triv-D},
$\pi_*(M'')\neq0$. The fibre sequence~\eqref{eq:M-fibseq}
induces a commutative diagram
\[
\xymatrix{
0\ar[r] &\pi_*(D\wedge_DM')\ar[r]\ar[d] & \pi_*(D\wedge_DM)\ar@{->>}[r]\ar[d]
                             & \pi_*(D\wedge_DM'')\ar[d]\ar@/^40pt/[dd]^{=} & \\
&\pi_*(D\wedge_AM')\ar[r] & \pi_*(D\wedge_AM)\ar@{->>}[r]
                            & \pi_*(D\wedge_AM'')\ar[d] & \\
                            &&&\pi_*(D\wedge_DM'')&
}
\]
in which a non-zero element $x\in\pi_*(D\wedge_DM'')$ lifts
to $\pi_*(D\wedge_DM)$ and so is in the image of composition
passing through $\pi_*(D\wedge_AM)$. Therefore
$\pi_*(D\wedge_AM)\neq0$.
\end{proof}

\section{Lubin-Tate cohomology of classifying spaces}
\label{sec:E^*BG}

We will denote by $E$ any Lubin-Tate spectrum such as $E_n$ or
$\Enr_n$, and then $K$ will denote the corresponding version of
Morava $K$-theory see~\cite{AB&BR:Enr} for details. The spectrum
$E$ is a commutative $S$-algebra, while $K$ is an $E$-algebra in
the sense of~\cite{EKMM}. The homotopy groups $\pi_*E$ and $\pi_*K$
are $2$-periodic and $\pi_0E$ is Noetherian; $\pi_0K$ is a field,
although $K$ is only homotopy commutative if~$p$ is an odd prime,
while when $p=2$ it is not even that. Nevertheless, we will view
$K$ as a kind of `topological division ring'.

The following lemma will allows us in certain circumstances to
relate modules over $E^{BG}=F(BG_+,E)$ to modules over
$K^{BG}=F(BG_+,K)$.
\begin{lem}\label{lem:E-K}
For any $E^{BG}$-module $M$, there is isomorphism of $K$-modules
\begin{equation*}
K\wedge_{E^{BG}} M
      \iso (K\wedge_E E)\wedge_{K\wedge_E E^{BG}}(K\wedge_E M).
\end{equation*}
In particular, there is an isomorphism of $K$-modules
\[
K\wedge_{E^{BG}} E \iso K\wedge_{K^{BG}} K.
\]
\end{lem}

\begin{proof}
This follows from an obvious generalization
of~\cite[proposition~III.3.10]{EKMM}. Since there are isomorphisms
of $E$-algebras $K\iso K\wedge_E E$ and $K^{BG}\iso K\wedge_E E^{BG}$,
for any $E^{BG}$-module $M$,
\begin{align*}
K\wedge_{E^{BG}} M &\iso K\wedge_E(E\wedge_{E^{BG}} M) \\
                   &\iso (K\wedge_K K)\wedge_E (E\wedge_{E^{BG}} M) \\
                   &\iso (K\wedge_E E)\wedge_{K\wedge_E E^{BG}}(K\wedge_E M).
\qedhere
\end{align*}
\end{proof}
\begin{rem}\label{rem:transfer}
By a standard argument making use of the Becker-Gottlieb
transfer~\cite{JB&DG:transfer}, after $p$-localization,
$\Sigma^\infty BG_+$ is a retract of $\Sigma^\infty BG'_+$
where $G'$ is any $p$-Sylow subgroup of~$G$. In particular,
when $p\nmid|G|$ we have
\[
F(BG_+,E) \sim E, \quad F(BG_+,K) \sim K.
\]
\end{rem}
\begin{thm}\label{thm:E-BG}
Let $G$ be a finite group.  \\
\emph{(a)} The $K$-cohomology $K^*(BG_+)$ is a finite dimensional
$K^*$-vector space and the $E$-cohomology $E^*(BG_+)$ is a finitely
generated $E^*$-module. \\
\emph{(b)} If $K^*(BG_+)$ is concentrated in even degrees, then
$E^*(BG_+)$ is a free $E^*$-module of finite rank and
\[
K^*(BG_+) = K^*\otimes_{E^*}E^*(BG_+)
          = E^*(BG_+)/\mathfrak{m}E^*(BG_+).
\]
\emph{(c)} $K^*(BG_+)$ is an augmented Artinian local\/ $K^*$-algebra
whose maximal ideal is nilpotent. Hence $E^*(BG_+)$ is an augmented
pro-Artinian local $E^*$-algebra,
\[
E^*(BG_+) = \lim_r E^*(BG_+)/\mathfrak{m}^rE^*(BG_+).
\]
\end{thm}
\begin{proof}
(a) See~\cite{HKR:K*BG,HKR:GenChars} for example. \\
(b) See~\cite[proposition~2.5]{MH&NS}. \\
(c) Following Remark~\ref{rem:transfer}, we can reduce
to the case where $G$ is a $p$-group using the transfer
associated with a $p$-Sylow subgroup $G'\leq G$. The
case of a cyclic $p$-group $C_{p^r}$ is well known and
\[
K^*({BC_{p^r}}_+) = K^*[y]/(y^{p^r}).
\]
The case of a general $p$-group $G$ of order $p^m$ follows
by induction on $m$ since there is always a normal subgroup
$N\ideal G$ of index~$p$ and this permits an argument with
the Serre spectral sequence associated with the fibration
\[
BN \lra BG \lra BC_p
\]
as used in~\cite{DCR:K^*BG} to calculate $K^*(BG_+)$ from
knowledge of $K^*(BN_+)$ as input.
\end{proof}

It is known that $K^*(BG_+)$ need not be concentrated in
even degrees~\cite{IK&KL}.

We are interested in the $E$-algebras $E^{BG}=F(BG_+,E)$
and $K^{BG}=F(BG_+,K)$, each of which is $K$-local. Of
course the diagonal $BG\lra BG\times BG$ induces the
product on each of these, but only $E^{BG}$ is strictly
commutative, while $K^{BG}$ is homotopy commutative when
$p\neq2$ and merely associative when $p=2$. At the level
of homotopy groups, $E^*(BG_+)=\pi_*(E^{BG})$ and
$K^*(BG_+)=\pi_*(K^{BG})$ are both graded commutative.

Now we can apply our earlier results to give
\begin{thm}\label{thm:faithfulness-EBG}
For any finite group $G$, $E$ and $K$ are faithful
$E^{BG}$-modules in the $K$-local category.
\end{thm}
\begin{proof}
It suffices to show that $K$ is faithful. By Lemma~\ref{lem:E-K},
for any $E^{BG}$-module there is an isomorphism
\begin{align*}
K\wedge_{E^{BG}} M \iso (K\wedge_E E)\wedge_{K\wedge_E E^{BG}}(K\wedge_E M).
\end{align*}
The natural morphism of $E$-algebras
\begin{equation*}
K\wedge_E F(BG_+,E) \lra F(BG_+,K\wedge_E E)
\end{equation*}
is a weak equivalence since $K$ is a finite cell $E$-module,
so by \cite[theorem~III.4.2]{EKMM} it is enough to know that
\begin{equation*}
(K\wedge_E E)\wedge_{K^{BG}}(K\wedge_E M)
                \iso K\wedge_{K^{BG}}(K\wedge_E M) \nsim *.
\end{equation*}
If $M$ is $K$-local and non-trivial, then
$K\wedge_{K^{BG}}(K\wedge_E M) \nsim *$,
because we know from Theorem~\ref{thm:Non-triv} that $K$ is
faithful as a $K^{BG}$-module.
\end{proof}

\section{Galois theory and $E^{BG}$}\label{sec:E^BG&GaloisThy}

In this section we will consider extensions of the form
\[
E^{BG} = F(BG_+,E) \lra F(EG_+,E) \sim E
\]
with $G$ a finite group and consider whether or not they
are Galois. Since we know they are faithful, the issue is
whether such an extension satisfies the unramified condition
that the map
\[
\Theta\:F(BG_+,E)\wedge_{E^{BG}}F(BG_+,E) \lra F(G_+,E)
\]
is weak equivalence, and therefore there is a weak equivalence
\begin{equation}\label{eq:E^BG-unram}
E \wedge_{E^{BG}} E \xrightarrow{\;\sim\;} \prod_G E.
\end{equation}
In particular, this condition implies that $\pi_*(E\wedge_{E^{BG}}E)$
is concentrated in even degrees.

We begin by considering the case of cyclic $p$-groups $C_{p^r}$.
\begin{thm}\label{thm:E^BCp^r-ramified}
For each $r\geq1$, the extension
\[
E^{BC_{p^r}} = F({BC_{p^r}}_+,E) \lra F({EC_{p^r}}_+,E)
\]
is ramified and hence it is not\/ $C_{p^r}$-Galois.
\end{thm}
\begin{proof}
We recall (see for example~\cite[lemma~5.1]{HKR:GenChars})
that
\begin{equation*}
(E^{BC_{p^r}})_* = E^*[[y]]/([p^r]y),
\end{equation*}
where $y\in(E^{BC_{p^r}})_0 = E^0({BC_{p^r}}_+)$ and the
$p$-series $[p]y$ has the form
\begin{equation*}
[p]y \equiv y^{p^{n}} \mod{\mathfrak{m}},
\end{equation*}
so for each $r\geq1$ the $p^r$-series is inductively
defined by
\begin{align*}
[p^r]y = [p]([p^{r-1}]y) &= p^ry + \cdots + y^{p^{rn}} + \cdots \\
 &\equiv y^{p^{rn}} \mod{\mathfrak{m}}.
\end{align*}
By the Weierstrass preparation theorem, there is a polynomial
\begin{equation*}
\<p^r\>y = p^r + \cdots + y^{p^{rn}-1}
         \equiv y^{p^{rn}-1} \mod{\mathfrak{m}}
\end{equation*}
for which
\[
[p^r]y = y\<p^r\>y(1 + y f_r(y)),
\]
where $f_r(y)\in E^*[[y]]$. Then we have
\begin{equation*}
(E^{BC_{p^r}})_* = E^*[[y]]/(y\<p^r\>y).
\end{equation*}

The $(E^{BC_{p^r}})_*$-module $E_*$ admits the periodic
minimal free resolution
\begin{equation}\label{eq:E^*Cpr-res}
0\la E_* \xleftarrow{\ph{\;y\;}} (E^{BC_{p^r}})_*
\xleftarrow{\;y\;}(E^{BC_{p^r}})_* \xleftarrow{\<p^r\>y}(E^{BC_{p^r}})_*
\xleftarrow{\;y\;}(E^{BC_{p^r}})_* \xleftarrow{\<p^r\>y}(E^{BC_{p^r}})_*
\xleftarrow{\ph{\;y\;}}\ldots,
\end{equation}
so $\Tor^{(E^{BC_{p^r}})_*}_{*,*}(E_*,E_*)$ is the homology
of the complex
\begin{multline*}
0\la E_*\otimes_{(E^{BC_{p^r}})_*}(E^{BC_{p^r}})_*
\xleftarrow{\;I\otimes y\;}E_*\otimes_{(E^{BC_{p^r}})_*}(E^{BC_{p^r}})_*
\xleftarrow{I\otimes \<p^r\>y}E_*\otimes_{(E^{BC_{p^r}})_*}(E^{BC_{p^r}})_* \\
\xleftarrow{\;I\otimes y\;}E_*\otimes_{(E^{BC_{p^r}})_*}(E^{BC_{p^r}})_*
\xleftarrow{I\otimes \<p^r\>y}E_*\otimes_{(E^{BC_{p^r}})_*}(E^{BC_{p^r}})_*
\xleftarrow{\ph{\;I\otimes y\;}}\ldots,
\end{multline*}
which is equivalent to
\begin{equation}\label{eq:E^*Cpr-Torcomplex}
0\la E_*\xleftarrow{\;0\;}E_*\xleftarrow{p^r}E_*
\xleftarrow{\;0\;}E_*\xleftarrow{p^r}E_*
\xleftarrow{\ph{\;0y\;}}\ldots.
\end{equation}
Since $E_*$ is torsion-free, for $s\geq0$ this gives
\begin{equation}\label{eq:E^*Cpr-Tor}
\Tor^{(E^{BC_{p^r}})_*}_{s,*}(E_*,E_*) =
\begin{cases}
\ph{abc} E_* & \text{if $s=0$}, \\
E_*/p^rE_*   & \text{if $s$ is odd}, \\
\ph{abc} 0   & \text{otherwise}.
\end{cases}
\end{equation}

Thus in the K\"unneth spectral sequence
\begin{equation}\label{eq:KSS-Cp^r}
\mathrm{E}^2_{s,t} = \Tor^{(E^{BC_{p^r}})_*}_{s,t}(E_*,E_*)
\Lra \pi_{s+t} (E \wedge_{E^{BC_{p^r}}} E)
\end{equation}
there can be no non-trivial differentials since for degree
reasons the only possibilities involve $E_*$-module
homomorphisms of the form
\[
d^{2k-1}\:\mathrm{E}^2_{2k-1,t} = E_t/p^rE_t
             \lra \mathrm{E}^2_{0,t+2k-2} = E_{t+2k-2},
\]
with torsion-free target. This shows that the odd degree
terms in $\pi_*(E\wedge_{E^{BC_{p^r}}}E)$ are not zero,
contradicting the unramified condition~\ref{eq:E^BG-unram}
for a Galois extension.
\end{proof}

\begin{rem}\label{rem:E^BCp^r-ramified}
If we work rationally, then the K\"unneth spectral sequence
\[
\mathrm{E}^2_{s,t}(C_{p^r};\Q)
   = \Tor^{((E^{BC_{p^r}})\Q)_*}_{s,t} (E_*\Q,E_*\Q)
   \Lra \pi_{s+t}(E\Q \wedge_{(E^{BC_{p^r}})\Q} E\Q)
\]
has $\mathrm{E}^2_{s,*}(C_p^r;\Q) = 0$ except when $s=0$,
giving
\[
\pi_*(E\Q \wedge_{(E^{BC_{p^r}})\Q} E\Q)
     = E_*\Q \otimes_{(E^{BC_{p^r}})_*\Q} E_*\Q.
\]
This shows that higher filtration terms in the
K\"unneth spectral sequence~\ref{eq:KSS-Cp^r}
contribute $p$-torsion.
\end{rem}

Now we extend Theorem~\ref{thm:E^BCp^r-ramified} to
arbitrary $p$-groups.
\begin{thm}\label{thm:E^BG-ramified}
Let $G$ be a non-trivial $p$-group. Then the extension
\[
F(BG_+,E) \lra F(EG_+,E)
\]
is not $G$-Galois. More precisely, this extension is
ramified:
\[
F(EG_+,E) \wedge_{F(BG_+,E)} F(EG_+,E) \nsim \prod_{G}F(EG_+,E).
\]
\end{thm}
\begin{proof}
Choose a non-trivial epimorphism $G\lra C_p$; then for
some $k\geq 1$ there is a factorization
\smallskip
\begin{equation}\label{eq:E^BG-ramified}
\xymatrix{
C_{p^k}\ar@{ >->}[r]\ar@/^12pt/@{->>}[rr] &G\ar@{->>}[r] &C_p
}
\end{equation}
inducing morphisms between the associated K\"unneth spectral
sequences
\begin{equation}\label{eq:KSS-comparisons}
\mathrm{E}^r_{**}(C_p)\lra \mathrm{E}^r_{**}(G)
                          \lra \mathrm{E}^r_{**}(C_{p^k}).
\end{equation}
As we saw in the proof of Theorem~\ref{thm:E^BCp^r-ramified},
the two outer spectral sequences have trivial differentials.
We will analyze the composite morphism
$\mathrm{E}^2_{**}(C_p)\lra\mathrm{E}^2_{**}(C_{p^k})$.

On choosing generators appropriately, the canonical epimorphism
$C_{p^k}\lra C_p$ induces the $E_*$-algebra monomorphism
\[
(E^{BC_p})_* = E_*[[y]]/([p]y) \lra (E^{BC_{p^k}})_* = E_*[[y]]/([p^k]y);
                                     \quad y \mapsto [p^{k-1}]y,
\]
hence the induced map between the two resolutions of the
form~\eqref{eq:E^*Cpr-res} is
\begin{equation*}
\xymatrix{
0 & E_*\ar[d]_{=}\ar[l] & \ar[l] (E^{BC_{p}})_*\ar@<-1ex>[d]^{\rho_0}
     && \ar[ll]_(0.5){\;y\;} (E^{BC_{p}})_* \ar[d]^{\rho_1}
     &&  \ar[ll]_{\;\<p\>y\;} (E^{BC_{p}})_*\ar[d]^{\rho_2}  & \ar[l]_(0.3){\;y\;} \cdots \\
0 & E_*\ar[l] & \ar[l] (E^{BC_{p^k}})_*  &&  \ar[ll]_(0.5){\;y\;} (E^{BC_{p^k}})_*
     &&  \ar[ll]_{\;\<p^k\>y\;} (E^{BC_{p^k}})_* & \ar[l]_(0.3){\;y\;} \cdots
}
\end{equation*}
where the vertical maps are given by
\[
\rho_{2s}\:g(y)\mapsto g([p^{k-1}]y),
\quad
\rho_{2s-1}\:h(y)\mapsto h([p^{k-1}]y)\<p^{k-1}\>y.
\]
Applying $E_*\otimes_{(E^{BC_{p^r}})_*}(-)$ to the first
and second rows with $r=1$ and $k$ respectively, we obtain
a map of chain complexes
\begin{equation*}
\xymatrix{
0 & \ar[l] E_*\ar@<0ex>[d]^{\rho'_0}_{=}
     && \ar[ll]_(0.5){\;0\;} E_* \ar[d]^{\rho'_1=p^{k-1}\.}
     &&  \ar[ll]_{\;p\;} E_*\ar[d]^{\rho'_2}_{=} & &\ar[ll]_(0.5){\;0\;} \cdots \\
0 & \ar[l] E_*  &&  \ar[ll]_(0.5){\;0\;} E_*
     &&  \ar[ll]_{\;p^k\;} E_* & &\ar[ll]_(0.5){\;0\;} \cdots
}
\end{equation*}
where
\[
\rho'_{2s}=\id,
\quad
\rho'_{2s-1}=p^{k-1}\.\;.
\]
Applying this to the odd degree terms given in~\eqref{eq:E^*Cpr-Tor}
we see that the induced map
\[
E_*/pE_* \xrightarrow{\;p^{k-1}\.\;} E_*/p^kE_*
\]
is always a monomorphism. Therefore in~\eqref{eq:KSS-comparisons},
the first of the induced morphisms
\[
\mathrm{E}^2_{**}(C_p)\lra \mathrm{E}^r_{**}(G)
                          \lra \mathrm{E}^r_{**}(C_{p^k})
\]
is a monomorphism. There can be no higher differentials killing
elements in its image because they map to non-trivial elements
of $\mathrm{E}^2_{**}(C_{p^k})$ which survive the right hand
spectral sequence. This shows that $\mathrm{E}^\infty_{**}(G)$
contains elements of odd degree, and as in the cyclic group
case this is incompatible with the unramified condition.
\end{proof}

We can extend this result to the class of $p$-nilpotent groups.
A finite group $G$ is $p$-nilpotent if one and hence each $p$-Sylow subgroup
$P\leq G$ has a normal $p$-complement, \ie, there is a normal
subgroup $N\ideal G$ with $p\nmid |N|$ and $G=PN=P\ltimes N$.
A convenient summary of the properties of such groups can be
found in~\cite[section~7]{JM:p-nilpotent}, see also~\cite{DJSR:Book}.

\begin{cor}\label{cor:E^BG-ramified}
If $G$ is a $p$-nilpotent group for which $p$ divides $|G|$,
then the extension
\[
F(BG_+,E) \lra F(EG_+,E)
\]
is ramified and so is not $G$-Galois.
\end{cor}
\begin{proof}
By a result of Tate~\cite{Tate:nilpotentquot}, $G$ being
$p$-nilpotent is equivalent to the restriction homomorphism
giving an isomorphism
\begin{equation*}
\res^G_P\:H^*(BG;\F_p) \xrightarrow{\;\iso\;} H^*(BP;\F_p),
\end{equation*}
and in fact it is sufficient that this holds in degree~$1$.
Comparison of the Serre spectral sequences for $K^*(BG_+)$
and $K^*(BP_+)$ shows that
\begin{equation*}
K^*(BG_+) \xrightarrow{\;\iso\;} K^*(BP_+).
\end{equation*}
It now follows that
\begin{equation*}
E^*(BG_+) \xrightarrow{\;\iso\;} E^*(BP_+).
\end{equation*}
and the result can be deduced from Theorem~\ref{thm:E^BG-ramified}.
\end{proof}
\begin{rem}\label{rem:p-nilpotent}
The condition of $G$ being a $p$-nilpotent group should not
be confused with the condition that the conjugation action
of $G$ on $\F_p[G]$ is nilpotent. The latter is used
in~\cite[proposition~5.6.3]{JR:Opusmagnus} to ensure convergence
of the Eilenberg-Moore spectral sequence and so to prove that
for such groups
\[
F(BG_+,H\F_p) \lra F(EG_+,H\F_p)
\]
is a $G$-Galois extension. The example of $G=\Sigma_3$, the
third symmetric group, for the prime $p=2$ illustrates this.
For each of the Sylow $2$-subgroups
\begin{equation*}
\{\id,(1,2)\},\;\{\id,(1,3)\},\;\{\id,(2,3)\}
\end{equation*}
has as normal complement
\begin{equation*}
N = \{\id,(1,2,3),(1,3,2)\},
\end{equation*}
therefore $\Sigma_3$ is $2$-nilpotent. However, the
$\Sigma_3$-module $\F_2[\Sigma_3]$ contains the
$2$-dimensional non-trivial simple submodule
\begin{equation*}
V = \{x(1,2)+y(1,3)+z(2,3):x+y+z=0\},
\end{equation*}
so by Jordan-H\"older theory every composition
series for $\F_2[\Sigma_3]$ must have this as a
composition factor. Hence the action of $\Sigma_3$
on $\F_2[\Sigma_3]$ cannot be nilpotent.
\end{rem}

\section{Some observations on the Eilenberg-Moore
                    spectral sequence}\label{sec:EMSS}

In \cite[section~5.6]{JR:Opusmagnus}, it is shown that
for a finite $p$-group $G$, the Eilenberg-Moore spectral
sequence with
\begin{equation}\label{eq:EMSS-HFp}
\mathrm{E}^2_{s,t} = \Tor^{H^*(BG_+;\F_p)}_{s,t}(\F_p,\F_p)
\end{equation}
converges to $\pi_*(F(G_+,H\F_p)) = \pi_*(\prod_G\F_p)$.
By comparing it with the K\"unneth spectral sequence
for $\pi_*(H\F_p\wedge_{F(BG_+,H\F_p)}H\F_p)$, it
is also shown that
\[
F(BG_+,H\F_p) \lra F(EG_+,H\F_p)
\]
is a $G$-Galois extension.

Let us consider in detail the case $G=C_p$ for $p$ an odd
prime. The case when $p=2$ is similar. First we write
\[
H^*(BC_p) = H^*({BC_p}_+;\F_p) = \F_p[y]\otimes\Lambda(z),
\]
where $y\in H^2(BC_p)$ and $z\in H^1(BC_p)$. Then~\eqref{eq:EMSS-HFp}
becomes
\[
\mathrm{E}^2_{**} = \Gamma(\sigma z)\otimes\Lambda(\sigma y),
\]
where $\sigma y \in \mathrm{E}^2_{1,-2}$ and
$\sigma z \in \mathrm{E}^2_{1,-1}$ are the suspensions of $y$
and $z$, see~\cite{RW}. Writing $\gamma_r = \gamma_r(\sigma z)$.
The first non-trivial differential is
\[
d^{p-1}\gamma_p = \sigma y,
\]
and we have
\[
\mathrm{E}^p_{**} = \F_p[\zeta]/(\zeta^p)
    \otimes \Gamma(\gamma_{p^2})\otimes\Lambda(\gamma_p\sigma y),
\]
where $\zeta$ represents the class of $\sigma z$. The remaining
differentials are determined by the formulae
\[
d^{p^s-p^{s-1}-1}\gamma_{p^s} = \gamma_{p^{s-1}}\sigma y
\]
in
\[
\mathrm{E}^{p^s-p^{s-1}-1}_{**} = \F_p[\zeta]/(\zeta^p)
    \otimes \Gamma(\gamma_{p^s})\otimes\Lambda(\gamma_{p^{s-1}}\sigma y).
\]
Finally we have
\[
\mathrm{E}^\infty_{**} = \F_p[\zeta]/(\zeta^p),
\]
which is an avatar of $\prod_{C_p}\F_p$. These differentials
are forced by the known answer and multiplicativity, and are
also related to the discussion of~\cite[section~6]{RW}.
For Lubin-Tate theory $(E^{BC_{p^r}})_*$ is free over $E_*$
and the comparison of the Eilenberg-Moore with the K\"unneth
spectral sequence together with our
Theorems~\ref{thm:E^BCp^r-ramified} and~\ref{thm:E^BG-ramified}
has the following consequence.
\begin{prop} \label{prop:emen}
For the cyclic $p$-group $C_{p^r}$ the $E$-theory Eilenberg-Moore
spectral sequence for $BC_{p^r}$ with
\begin{equation*}
{}^{\text{\rm\tiny L-T}}\mathrm{E}^2_{s,t}
                             = \Tor^{(E^{BC_{p^r}})_*}(E_*,E_*)
\end{equation*}
does not converge to $\pi_*(\prod_{C_{p^r}}E)$.
\end{prop}

Just as in the $H\F_p$ case, we can compare the Morava $K$-theory
based Eilenberg-Moore spectral sequence with the K\"unneth spectral
sequence. Work of Bauer~\cite{TB:EMSS} on the convergence of the
$\Cotor$-version of this Eilenberg-Moore spectral sequence shows
that the corresponding spectral sequence converges for $G=C_p$
and odd primes~$p$, and therefore
\begin{equation*}
K \wedge_{K^{BC_p}} K \sim \prod_{C_p} K.
\end{equation*}
The extension of $S$-algebras $K^{BC_p} \lra K^{EC_p}$ can be
interpreted as a Galois extension of non-commutative $S$-algebras.

\end{document}